\documentclass[10pt,a4paper]{amsart}
\usepackage[utf8]{inputenc}
\usepackage{amsmath}
\usepackage{amsfonts}
\usepackage{amssymb}
\usepackage{amsthm}
\usepackage{comment}
\usepackage{color} 
\newtheorem{theorem}{Theorem}[section]

\newtheorem{lemma}{Lemma}[section]

\newtheorem{defi}{Definition}[section]
\newtheorem{remark}{Remark}[section]
\newtheorem{ass}{Assumption}[section]

\newcommand{\dive}{\text{\rm div}}

\newcommand{\dist}{\text{\rm dist}}

\newcommand{\R}{{\mathbb R}}
\newcommand{\C}{{\mathbb C}}
\newcommand{\N}{{\mathbb N}}

\begin{document}
	
	\title[Stable determination of a Lam\'e coefficient]{Stable determination of a Lam\'e coefficient by one internal measurement of displacement}
	
	\author[G.Di Fazio]{Giuseppe~Di Fazio}
	\address{Dipartimento di Matematica e Informatica\\ 
		Università di Catania\\
		Via\-le A. Doria 6, 95125, Catania, Italy}
	\email{difazio@dmi.unict.it}
	
	\author[E.Francini]{Elisa Francini}
	\address{Dipartimento di Matematica \\ 
		Università di Firenze\\
		Via\-le Morgagni 67 A, Firenze, Italy}
	\email{francini@math.unifi.it}
	
	\author[F.Raciti]{Fabio Raciti}
	\address{Dipartimento di Matematica e Informatica\\ 
		Università di Catania\\
		Via\-le A. Doria 6, 95125, Catania, Italy}
	\email{fraciti@dmi.unict.it}
		
	\author[S.Vessella]{Sergio Vessella}
	\address{Dipartimento di Matematica \\ 
		Università di Firenze\\
		Via\-le Morgagni 67 A, Firenze, Italy}
	\email{sergio.vessella@unifi.it}
	\begin{abstract}
In this paper we show that the shear modulus $\mu$ of an isotropic elastic body can be stably recovered by the knowledge of one single displacement field whose boundary data can be assigned independently of the unknown elasticity tensor.
\end{abstract}
	\maketitle
	
\section{Introduction}
In this paper we consider the following problem:
let $\Omega\subset\R^n$, $n\geq 2$ be a bounded domain representing an elastic isotropic body with Lam\'e coefficients $\lambda$ and $\mu$. Assuming $\lambda$ is known, we want to stably recover the shear modulus $\mu$ from the knowledge of one single displacement field in $\Omega$, that is  a solution $u\in (H^1(\Omega))^n$ to the elasticity system
\[\dive \left(\C\hat{\nabla}u\right)=0\mbox{ in }\Omega\]
where  
\[
\C \hat{\nabla} u = \lambda \dive(u) I_n + \mu \hat{\nabla}u,
\]
with $I_n$ being the $n\times n$ identity matrix and $\hat{\nabla}u=\frac{1}{2}(\nabla u+(\nabla u)^T)$.

This problem is connected to the imaging method usually called Elastography. The most common approach to Elastography consists in a 2-step reconstruction. In the first step the elastic displacement is recovered from either sound waves (Ultrasound Elastography) or protons' propagation (Magnetic Resonance Elastography). In the present analysis we study the second step, namely the quantitative reconstruction of Lam\'e parameters from the knowledge of elastic displacements.

This problem has recently been studied in \cite{BBIM} and \cite{Lai} for isotropic elasticity tensors in the time harmonic regime, and in \cite{BMU} for the anisotropic case at zero frequency. In these papers a key point in the proof of unique and stable reconstruction consists in looking for few displacement fields satisfying a rank maximality condition concerning their gradients. Existence of such non degenerate sets of solutions is usually proven by using density arguments (such as Runge approximation) or CGO solutions. Unfortunately, it is not possible to choose a-priori boundary data of these particular solutions, since they depend on the interior values of the unknown elasticity tensor. For an analysis of this issue in the scalar case we refer to \cite{Cap} and \cite{ABDC}.

For this reason, following the method used in \cite{A2014} and \cite{ADcFV2016}, we propose here the choice of boundary values in such a way that the degeneracy of $\nabla u$ can be controlled by quantitative estimates of unique continuation. We point out that this stability estimate is obtained with only one internal measurement.
On the other hand, here we focus our attention only on the shear modulus $\mu$ and assume that $\lambda$ is known. 
This is not a big restriction for the possible application of our result, because the shear modulus $\mu$ is the parameter  that changes more between healthy and damaged tissues (see \cite{50}).
  
The paper is organized as follows: in section \ref{assumptions} we list the main notations and assumptions; in section \ref{formulation} we formulate the problem and state our result. The main tools of our analysis are an integral stability estimate (section \ref{lemma}) and quantitative estimates of unique continuation (section \ref{QEUC}). Finally, section \ref{proof} contains the proof of our result. 
\section{Preliminary assumptions}\label{assumptions}
We denote points in $\R^{n}$ by $x=(x',x_n)$ where $x'\in \R^{n-1}$ and $x_n\in\R$. Analogously, we denote by $\Omega'$ the set of points in $\R^{n-1}$ such that $(x',x_n)$ belong to $\Omega$ for some $x_n$.
\begin{ass} \label{ipotesi-dominio}
We assume that $\Omega$, a bounded domain in $\R^n$ with $n\geq 2$, is a domain with $C^{1,1}$ boundary, that is, for any $x_0\in \partial \Omega$ there exists a rigid change of coordinates such that, $x_0=0$ and 
$$
\Omega \cap B_{r_0}(0) = 
\left\lbrace
x \in B_{r_0}(0) \ : \ x_n>\psi(x')
\right\rbrace. 
$$
where $\psi$ is a $C^{1,1}$ function defined in $B'_{r_0}(0)$ such that $\psi(0)=0$ and
$$
\|\psi\|_{C^{1,1}(B'_{r_0}(0))}
\leq 
M_0.
$$
\end{ass}

As usual, for any $d>0$, we set
$$
\Omega_d
=
\left\lbrace 
x \in \Omega \ : \ \dist(x,\partial \Omega)>d
\right\rbrace.
$$

In the sequel we deal with the Lam\'e coefficients $\lambda$ and $\mu$, on which we posit the following assumptions.

\begin{ass} \label{ipotesi-coeff}
$$
\begin{cases}
\mu,\lambda \in C^{0,1}(\bar{\Omega}) & \\
\|\mu\|_{C^{0,1}(\bar{\Omega})} + \|\lambda\|_{C^{0,1}(\bar{\Omega})}
\leq M
\end{cases}
\leqno(2a)
$$

$$
\mu(x) \geq \alpha_0 >0\,, \quad 2\mu(x)+n\lambda(x) \geq \beta_0>0 \
\text{\rm in} \ \bar{\Omega}. 
\leqno(2b-strong \ convexity)
$$
\end{ass}


We also assume 
\begin{ass}\label{ipotesi-g}
	The function $g$  belongs to $H^{3/2}(\partial \Omega)$, and
	\begin{equation}\label{normag}
	\|g\|_{H^{3/2}(\partial \Omega)}\leq L_0.
\end{equation}
Moreover we assume that $g$ is far from rigid movements, that is, given
\[\Theta(g):=\min \{\|g-(a+Wx)\|_{H^{1/2}(\partial \Omega)}\,:\, a\in\R^n,\, W\in \R^{n\times n}, W+W^T=0\}\]
we assume that 
\begin{equation}\label{norigid}
	\Theta(g)\geq \delta_0>0.
\end{equation}
\end{ass}

In the sequel we will use the following frequency of function $g$:
\begin{defi}
For any $g\in H^{1}(\partial \Omega)$, we set
$$
F[g]
=\dfrac{\|g\|_{H^{1}(\partial \Omega)}}{\Theta(g)}
$$
\end{defi}
For any function $g$ satisfying assumption \ref{ipotesi-g}, we have $F[g]\leq L_0\delta_0^{-1}$. 

\section{Formulation of the problem and main result}\label{formulation}
Let us consider functions $\mu_1$, $\mu_2$ and $\lambda$ such that $(\mu_1,\lambda)$ and $(\mu_2,\lambda)$ satisfy assumptions (2a) and (2b).

Let $u$ be the solution of the problem 
\begin{equation} \label{def-u}
\begin{cases}
\dive(\C_1 \hat{\nabla}u)=0 & \mbox{in } \Omega \\
u  = g & \mbox{on }\partial \Omega
\end{cases}
\end{equation}
where 
\[
\C_1 \hat{\nabla} u = \lambda \dive(u) I_n + \mu_1 \hat{\nabla}u,
\]
and let $v$ be the solution of the problem
 \begin{equation} \label{def-v}
 \begin{cases}
 \dive(\C_2 \hat{\nabla}v)=0 & \mbox{in } \Omega  \\
v = k & \mbox{on }\partial \Omega
 \end{cases}
 \end{equation}
 where 
 \[
 \C_2 \hat{\nabla} v = \lambda\dive(v) I_n + \mu_2 \hat{\nabla}v.
 \]

 Now we are ready to state our main result.
 \begin{theorem}\label{mainth}
 Let $d>0$ be such that $\Omega_d\neq \emptyset$. Then
 \begin{equation} \label{main-result}
\|\mu_1-\mu_2\|_{L^{\infty}(\Omega_d)}
\leq 
C
\left( 
\|\mu_1-\mu_2\|_{L^{\infty}(\partial\Omega)}
+
\|u-v\|^{1/4}_{L^{2}(\Omega)}
\right)^\delta
 \end{equation}
 where the constants $C>0$ and $\delta \in (0,1)$ depend only on $M_0$, $|\Omega|$, $r_0$, $d$, $M$, $\alpha_0$, $\beta_0$, $L_0$ and $\delta_0$.
 \end{theorem}

\section{An integral estimate}\label{lemma}
\begin{lemma} \label{giuseppe-fabio}
Let $u$ and $v$ be as in \eqref{def-u} and \eqref{def-v}. Moreover, let assumptions \ref{ipotesi-dominio} and \ref{ipotesi-coeff} hold true.
Then, there exists a positive constant $C$ depending only on $M_0$, $r_0$, $M$, $\alpha_0$, $\beta_0$, $\|g\|_{H^{3/2}(\partial\Omega)}$ and $\|k\|_{H^{3/2}(\partial\Omega)}$ such that
 \begin{equation} \label{DR1}
 \int_{\Omega} |\mu_1-\mu_2||\hat{\nabla} u|^2\,dx 
 \leq 
C\left( 
\|\mu_1-\mu_2\|_{L^\infty(\partial\Omega)} 
+
\|u-v\|^{1/4}_{L^2(\Omega)}
\right)\,.
 \end{equation}
\end{lemma}

In the sequel we will use the following notation. The dot between vectors is the scalar product while the dot between matrices is the product in the sense of  Frobenius.

\textit{Proof of Lemma \ref{giuseppe-fabio}.}

Let $u, v\in W^{1,2}(\Omega)$ be a solutions of \eqref{def-u} and \eqref{def-v}, respectively.

\noindent Let us set 
\[\varphi := \mu_1 -\mu_2,\quad \eta := \max_{\partial \Omega} |\mu_1 -\mu_2|.\]
By comparing the weak formulations of problems \eqref{def-u} and \eqref{def-v} we easily get
\begin{equation}\label{lambdauno}
\int_{\Omega} \varphi\hat{\nabla} u\cdot\hat{\nabla} \zeta \,dx = -\frac{1}{2} 
\int_{\Omega}\left(\lambda \dive (u-v) \dive \zeta + 2 \mu_2 \hat{\nabla} (u-v) \cdot \hat{\nabla} \zeta \right)\,dx\,,
\end{equation}
for every $\zeta \in W_{0}^{1,2}(\Omega)$.

To show the inequality we follow the route traced by \cite{A1986} by choosing a suitable test function. For $h>0$ set
\begin{equation}\label{zita}
\zeta(x)= \frac {\min\{(\varphi -\eta)^{+},h)\}}{h}\,u(x) 
\end{equation}
as our test function. 
We can easily check that
\[
\zeta(x)
=
\begin{cases}
\hfill 0 \hfill & \text{ se $\varphi \leq \eta$} 
\\
\dfrac{\varphi-\eta}{h}\,u  &\text{se $\eta<\varphi \leq \eta +h$} 
\\
\hfill u  \hfill  & \text{se $\varphi > \eta +h$}.
\end{cases}
\]
Let us first consider the LHS  of \eqref{lambdauno} with test function given by \eqref{zita}.
We have
\begin{equation}\label{LHS}
\begin{aligned}
\int_{\Omega} \varphi\hat{\nabla} u\cdot\hat{\nabla} \zeta \,dx=&\int_{\eta<\varphi<\eta +h}\!\!\!\!\! \varphi \hat{\nabla} u \cdot \hat{\nabla} \left(\frac{\varphi -\eta}{h} \, u\right) \,dx +  
\int_{\varphi>\eta +h}\!\!\!\!\! \varphi |\hat{\nabla} u |^2dx\\
=&\frac{1}{2h}\int_{\eta<\varphi<\eta +h} \!\!\!\!\!\varphi \hat{\nabla} u \cdot \left[\nabla\varphi \otimes u+u\otimes\nabla\varphi\right]dx
\\&+
\frac{1}{h}\int_{\eta<\varphi<\eta +h} \!\!\!\!\!\varphi(\varphi-\eta) |\hat{\nabla} u|^2 \,dx+  
\int_{\varphi>\eta +h} \!\!\!\!\!\varphi |\hat{\nabla} u |^2dx
\end{aligned}
\end{equation}
Let us now focus on  the integral
\begin{equation}\label{I}
I=\frac{1}{2h}\int_{\eta<\varphi<\eta +h} \!\!\!\!\!\varphi \hat{\nabla} u \cdot  \left[\nabla\varphi \otimes u+u\otimes\nabla\varphi\right]dx.
\end{equation}
By using the fact that, for any symmetric matrix $A$ and for any vector $b$ anc $c$ 
$A\cdot [b\otimes c+c\otimes b]= 2b^T Ac$, 
\begin{equation}\label{triplo} 
\begin{aligned}
\frac{1}{2}\varphi \hat{\nabla} u \cdot \left[\nabla\varphi \otimes u+u\otimes\nabla\varphi\right]&= (\nabla \varphi)^T \varphi u\hat{\nabla} u =
\frac{1}{2} (\nabla \varphi^2)^T \hat{\nabla} u\,u\\
&=\frac{1}{2}\dive[\varphi^2  \hat{\nabla} u\, u ] -\frac{1}{2}  \varphi^2 \dive (\hat{\nabla}u\, u)\\
&= \frac{1}{2}\dive[\varphi^2  u \hat{\nabla} u ] -\frac{1}{2}  \varphi^2 \dive (u\hat{\nabla}u).
\end{aligned}
\end{equation}
Let us denote by $\nu_c$ the unit outer normal to the set $\{\varphi>c\}$ and us apply twice Green formula to obtain
\begin{align*}
I
=&
\frac{1}{2h} \left\{\eta^2 \int_{\varphi=\eta} (u\hat{\nabla} u)\cdot\nu_\eta \,ds_x -
 (\eta+h)^2 \int_{\varphi=\eta+h} (u\hat{\nabla} u )\cdot\nu_{\eta+h} \,ds_x \right\}
\\&\hphantom{spazio}-\frac{1}{2h}\int_{\eta<\varphi<\eta +h} \varphi^2 \dive (u\hat{\nabla} u) \,dx
\\
=&
\frac{1}{2h}\left\{ \eta^2 \int_{\varphi>\eta} \dive (u\hat{\nabla} u) \,dx -
 (\eta +h)^2 \int_{\varphi>\eta+h} \dive (u\hat{\nabla} u) \,dx \right\}
\\
&
\hphantom{spazio}-
\frac{1}{2h}\int_{\eta<\varphi<\eta +h}  \varphi^2 \dive (u\hat{\nabla} u)  \,dx 
\\
=&
-\frac{1}{2h}\int_{\eta<\varphi<\eta +h} (\varphi^2 -\eta^2) \dive (u\hat{\nabla} u) \,dx 
-
(\eta +h/2) \int_{\varphi>\eta+h} \dive (u\hat{\nabla} u)\,dx .
\end{align*}

It is easy to check that in the set $\{\eta<\varphi<\eta+h\}$ we have $0\leq \frac{\varphi^2-\eta^2}{2h}\leq \eta+\frac{h}{2}$,
and, hence,
\begin{equation} \label{3-6}
|I|
\leq (\eta +h/2) \int_{\varphi>\eta} |\dive (u\hat{\nabla} u)| \,dx 
\end{equation}

By putting together \eqref{LHS}, \eqref{I} and \eqref{3-6}, we have that, for $\zeta$ given by \eqref{zita}

\begin{equation}\label{stimaLHS}
\begin{aligned}
\int_{\Omega} \varphi\hat{\nabla} u\cdot\hat{\nabla} \zeta \,dx&\geq -|I| +\frac{1}{h}\int_{\eta<\varphi<\eta +h} \!\!\!\!\!\varphi(\varphi-\eta) |\hat{\nabla} u|^2 \,dx+  
\int_{\varphi>\eta +h} \!\!\!\!\!\varphi |\hat{\nabla} u |^2dx\\
&\geq - (\eta +h/2) \int_{\varphi>\eta} |\dive (u\hat{\nabla} u)| \,dx +\int_{\varphi>\eta +h} \!\!\!\!\!\varphi |\hat{\nabla} u |^2dx.
\end{aligned}
\end{equation}
Let us now estimate the RHS of \eqref{lambdauno} for $\zeta$ given by \eqref{zita}. 

We have
\begin{align}
\int_{\Omega}\lambda \dive (u-v) \dive \zeta \,dx
=&
\int_{\eta<\varphi<\eta +h} \lambda \dive (u-v) \left[ \frac{\varphi -\eta}{h} \,\dive (u) +\frac{\nabla \varphi \cdot u}{h} \right]\,dx
+
\notag \\
&
+
 \int_{\varphi>\eta +h} \lambda \dive (u-v)\dive (u)\,dx
 \notag
\end{align}
then, by assumption \ref{ipotesi-coeff},
\begin{equation}\label{lambdaotto}
\left|\int_{\Omega}\lambda \dive (u-v) \dive \zeta dx\right| 
\leq M\left(
\int_{\varphi >\eta} \!\!\!\!|\dive (u-v) \dive u|\,dx+
\frac{1}{h}  \int_{\varphi>\eta+h }\!\!\!\! \!\!\!\!|\dive (u-v)| |\nabla \varphi| |u|\,dx
\right) 
\end{equation}

We proceed in the same way to estimate
\begin{align*}
\int_{\Omega} 2 \mu_2 \hat{\nabla} (u-v) \cdot \hat{\nabla} \zeta \,dx
=&
\frac{1}{h} \int_{\eta<\varphi< \eta +h} \!\!\!\! \mu_2 \hat{\nabla} (u-v) \cdot [\nabla \varphi \otimes u +u\otimes \nabla \varphi
+
2(\varphi-\eta) \hat{\nabla} u] \,dx
\\
&
+ \int_{\varphi> \eta +h}\!\!\!\!2 \mu_2 \hat{\nabla} (u-v) \cdot \hat{\nabla} u\,dx
\\
=&
\frac{1}{h} \int_{\eta<\varphi< \eta +h} 2 \mu_2  (\nabla \varphi)^T \hat{\nabla} (u-v)\, u\,dx
+
\\
&
+
\frac{1}{h}\int_{\eta<\varphi< \eta +h} 2 \mu_2 (\varphi-\eta)\hat{\nabla} (u-v) \cdot
 \hat{\nabla} u\,dx
+
\\
&
+
\int_{\varphi> \eta +h} 2 \mu_2 \hat{\nabla} (u-v) \cdot \hat{\nabla} u\,dx.
\end{align*}

By assumption \ref{ipotesi-coeff}, we get
\begin{align} \label{lambdanove}
\left|\int_{\Omega} 2 \mu_2 \hat{\nabla} (u-v) \cdot \hat{\nabla} \zeta \,dx\right|
\leq &
\frac{2M}{h} \int_{\eta<\varphi< \eta +h}   |(\nabla \varphi)| |\hat{\nabla} (u-v)\, u|\,dx
\\
&
+
2M\int_{\varphi> \eta } |\hat{\nabla} (u-v) \cdot \hat{\nabla} u|\,dx
\notag
\end{align}

Finally, by putting together \eqref{stimaLHS}, \eqref{lambdauno}, \eqref{lambdaotto} and \eqref{lambdanove} we get

\begin{align*}
\int_{\varphi>\eta+h} \!\!\!\!\varphi|\hat{\nabla} u|^2\,dx
&
\leq
\left(\eta +\frac{h}{2}\right) \int_{\varphi>\eta} |\dive(u \hat{\nabla} u)| \,dx+
\\
&
+ \frac{M}{2}\left(
\int_{\varphi >\eta} \!\!\!\!|\dive (u-v) \dive u|\,dx+
\frac{1}{h}  \int_{\varphi>\eta+h }\!\!\!\! \!\!\!\!|\dive (u-v)| |\nabla \varphi| |u|\,dx
\right)\\
&+M\left(\int_{\varphi> \eta }\!\!\!\! |\hat{\nabla} (u-v) \cdot \hat{\nabla} u|\,dx+\frac{1}{h} \int_{\eta<\varphi< \eta +h} \!\!\!\!  |(\nabla \varphi)| |\hat{\nabla} (u-v)\, u|\,dx\right).
\end{align*}

If we use $-\varphi$ instead of $\varphi$ we found a similar estimate. Then, merging the two, we get
\begin{equation}\label{lambdadieci}
\begin{aligned}
&\int_{|\varphi|>\eta+h} \!\!\!\!|\varphi||\hat{\nabla} u|^2\,dx
\leq
\left(\eta +\frac{h}{2}\right) \int_{|\varphi|>\eta} |\dive(u \hat{\nabla} u)| \,dx+
\\
&
\hphantom{ppp}+ \frac{M}{2}\left(
\int_{|\varphi| >\eta} \!\!\!\!|\dive (u-v) \dive u|\,dx+
\frac{1}{h}  \int_{|\varphi|>\eta+h }\!\!\!\! \!\!\!\!|\dive (u-v)| |\nabla \varphi| |u|\,dx
\right)\\
&\hphantom{ppp}+M\left(\int_{|\varphi|> \eta }\!\!\!\! |\hat{\nabla} (u-v) \cdot \hat{\nabla} u|\,dx+\frac{1}{h} \int_{\eta<|\varphi|< \eta +h} \!\!\!\!  |(\nabla \varphi)| |\hat{\nabla} (u-v)\, u|\,dx\right)
\end{aligned}
\end{equation}
Since
$$
\int_{|\varphi|<\eta+h} |\varphi||\hat{\nabla} u|^2\,dx 
\leq
(\eta +h)\int_{|\varphi|<\eta+h} |\hat{\nabla} u|^2\,dx
\leq
(\eta +h)\int_{\Omega} |\hat{\nabla} u|^2\,dx
$$
we obtain, from \eqref{lambdadieci}
\begin{equation}\label{st}
\begin{aligned}
		\int_{\Omega} |\varphi||\hat{\nabla} u|^2\,dx
	\leq&
	(\eta +h)\int_{\Omega} |\hat{\nabla} u|^2 \,dx+\left(\eta +\frac{h}{2}\right) \int_{\Omega} |\dive(u \hat{\nabla} u)| \,dx+
	\\
	&+ \frac{M}{2}\left(
\int_{\Omega} |\dive (u-v) \dive u|\,dx+
\frac{1}{h}  \int_{\Omega }|\dive (u-v)| |\nabla \varphi| |u|\,dx
\right)\\
&+M\left(\int_{\Omega} |\hat{\nabla} (u-v) \cdot \hat{\nabla} u|\,dx+\frac{1}{h} \int_{\Omega} |\nabla \varphi| |\hat{\nabla} (u-v)\, u|\,dx\right)
\end{aligned}\end{equation}
By 
\cite[Theorem 7.1, chap. 3]{V}, we have
\[
\|u\|_{H^2(\Omega)}\leq C \|g\|_{H^{3/2}(\partial \Omega)}\]
and, since $|\nabla\varphi|\leq 2M$ by \eqref{ipotesi-coeff},
by H\"older inequality, we can easily get from \eqref{st} that
\begin{equation}\label{st2}\begin{aligned}
		\int_{\Omega} |\varphi||\hat{\nabla} u|^2\,dx
	\leq&C\|g\|_{3/2}\left\{(3\eta +2h)\|g\|_{3/2}\vphantom{\frac{2M}{h}}	\right.\\
	&\left.+ 2M\left(1+\frac{2M}{h}\right)\left(\|\dive (u-v)\|_{L^2(\Omega)}+\|\hat{\nabla} (u-v)\|_{L^2(\Omega)}\right)\right\}
\end{aligned}\end{equation}
where 
$\|g\|_{3/2}=\|g\|_{H^{3/2}(\partial\Omega)}$.

Let us notice that, by using an interpolation inequality (see, for example \cite[Theorem 7.28]{GT}), if we denote by $D$ any partial derivative of first order, we have
\begin{equation}\label{intL2}
\begin{aligned}
\|D (u-v)\|_{L^2(\Omega)}
&\leq 
 \varepsilon \|u-v\|_{H^2(\Omega)} + \dfrac{c}{\varepsilon} \|u-v\|_{L^2(\Omega)}
\\
&
\leq 
C 
\left[  \varepsilon (\|g\|_{3/2}+\|k\|_{3/2}) + \dfrac{1}{\varepsilon} \|u-v\|_{L^2(\Omega)}\right] 
\end{aligned}\end{equation}
Now choose $\varepsilon = \|u-v\|_{L^2(\Omega)}^{1/2}$ and we get
\begin{equation}\label{stimaderivdifferenza}
\|D (u-v)\|_{L^2(\Omega)} 
\leq 
C 
\left(\|g\|_{3/2}+\|k\|_{3/2}+1\right)\|u-v\|_{L^2(\Omega)}^{1/2},
\end{equation}
with $C$ depending only on $M_0$, $r_0$, $M$, $\alpha_0$, $\beta_0$. 

By \eqref{st2} and \eqref{stimaderivdifferenza} we have
\begin{equation}\label{st3}\begin{aligned}
		\int_{\Omega} |\varphi||\hat{\nabla} u|^2\,dx
	\leq&C\|g\|_{3/2}\left\{(3\eta +2h)\|g\|_{3/2}\vphantom{\left(1+\frac{M}{h}\right)}\right.
	\\
	&\left.+ 3M\left(1+\frac{M}{h}\right)\left(\|g\|_{3/2}+\|k\|_{3/2}+1\right)\|u-v\|_{L^2(\Omega)}^{1/2} 
	\right\}
\end{aligned}\end{equation}
We finally choose $h=\|u-v\|^{1/4}_{L^2(\Omega)}$
and get \eqref{DR1}.\qed

\section{Quantitative estimates of unique continuation}\label{QEUC}
Most of the results that we state and prove in this section are already known for solutions of the  Lam\'e system, but they are usually stated in terms of Neumann boundary conditions. We want here to use Dirichlet boundary conditions  that are better related to the internal measurements we are going to use. 

\begin{theorem}[Three sphere inequality for $|\hat{\nabla}u|$]\label{3sferegrad}
Under assumption \ref{ipotesi-coeff}, there exists $\theta\in(0,1]$ depending only on $\alpha_0$, $\beta_0$ and $M$ such that for every $u\in H^1(B_R)$ solution to the equation
\[\dive(\C \hat{\nabla}u)=0\]
and for every $r_1$, $r_2$, $r_3$ such that $0<r_1<r_2<r_3<\theta R$ we have
\begin{equation}\label{3sf}
\int_{B_{r_2}}|\hat{\nabla}u|^2dx\leq C\left(\int_{B_{r_1}}|\hat{\nabla}u|^2dx\right)^{\delta}\left(\int_{B_{r_3}}|\hat{\nabla}u|^2dx\right)^{1-\delta},
\end{equation}
where $C>0$ and $\delta\in(0,1)$ depend only on $\alpha_0$, $\beta_0$, $M$, $r_1/r_3$ and $r_2/r_3$. 
\end{theorem}
\begin{proof}
The proof of this estimates goes along the same lines of the proof of Corollary 3.3 in \cite{AMR2004}. The regularity of the Lam\'e coefficients can be lowered by starting from the three spheres inequality for the solution proved in \cite[Theorem 1.1]{LNW10}.
\end{proof}
\begin{theorem}[Lipschitz Propagation of Smallness]\label{LPSmallness}
Under assumptions \ref{ipotesi-dominio}, \ref{ipotesi-coeff} and \ref{ipotesi-g}, let $u\in H^1(\Omega)$ be solution to \eqref{def-u}.
Then, for every $\rho>0$ and for every $x\in\Omega_{5\rho}$, we have
\begin{equation}\label{LPS}
\int_{B_\rho(x)}|\hat{\nabla}u|^2dx\geq C_\rho\int_\Omega |\hat{\nabla}u|^2dx,
\end{equation}
where $C_\rho$ depends on $\alpha_0$, $\beta_0$, $M$, $r_0$, $M_0$, $|\Omega|$, $F[g]$, and $\rho$.
\end{theorem}
\begin{proof}
The proof follows essentially the same lines of the proof of Proposition 4.1 in \cite{AMR2004}. 
First of all, as in Lemma 4.2 in \cite{AMR2004}, by H\"older inequality and Sobolev inequality we can estimate
\begin{equation}\label{lemma4.2}
\int_{\Omega\setminus\Omega_{5\rho/8}} |\hat{\nabla} u|^2dx \leq C\rho^{1/n}\|g\|^2_{H^1(\partial\Omega)}.\end{equation}
The only difference consists in substituting inequality (4.6) in \cite{AMR2004} with the trace estimate
\[\|u\|_{H^{3/2}(\Omega)}\leq C \|g\|_{H^1(\partial\Omega)}.\]
As in (4.12) in \cite{AMR2004}, by using a suitable chain of balls and the three spheres inequality (Theorem \ref{3sferegrad}) we get 
\begin{equation}\label{31}
\frac{\|\hat{\nabla}u\|_{L^2(\Omega_{5\rho/8})}}{\|\hat{\nabla}u\|_{L^2(\Omega)}}\leq \frac{C}{\rho^{n/2}}\left(\frac{\|\hat{\nabla}u\|_{L^2(B_\rho(x))}}{\|\hat{\nabla}u\|_{L^2(\Omega)}}\right)^{\delta^L}
\end{equation}
where $C$ and $\delta$ depends only on $\alpha_0$, $\beta_0$, $M$ and $|\Omega|$, whereas $L\leq \frac{|\Omega|}{\omega_n\rho^n}$.

By \eqref{lemma4.2}, we have
\[\frac{\|\hat{\nabla}u\|_{L^2(\Omega_{5\rho/8})}}{\|\hat{\nabla}u\|_{L^2(\Omega)}}=1-\frac{\|\hat{\nabla}u\|_{L^2(\Omega\setminus\Omega_{5\rho/8})}}{\|\hat{\nabla}u\|_{L^2(\Omega)}}\geq 1-\frac{C\rho^{1/n}\|g\|_{H^1(\partial\Omega)}}{\|\hat{\nabla}u\|_{L^2(\Omega)}}.\]

Now, we need to estimate $\|\hat{\nabla}u\|_{L^2(\Omega)}$ from below. Let us set
\begin{equation}\label{medie}\overline{a}=\frac{1}{|\Omega|}\int_\Omega u\,dx,\quad \overline{W}=\frac{1}{|\Omega|}\int_\Omega\hat{\nabla}u \,dx.\end{equation} 
By trace inequality and Korn inequality we have
\begin{equation}\label{servedopo}
\|g-(\overline{a}+\overline{W}x)\|_{H^{1/2}(\partial\Omega)}
\leq C\|u-(\overline{a}+\overline{W}x)\|_{H^{1}(\Omega)}\leq C\|\hat{\nabla} u\|_{L^2(\Omega)},
\end{equation}
 and, hence,
\begin{equation}\label{e}
\frac{\|\hat{\nabla}u\|_{L^2(\Omega_{5\rho/8})}}{\|\hat{\nabla}u\|_{L^2(\Omega)}}\geq 1-\frac{C\rho^{1/n}\|g\|_{H^1(\partial\Omega)}}{\|g-(\overline{a}+\overline{W}x)\|_{H^{1/2}(\partial\Omega)}}\geq 1-C\rho^{1/n}F[g].\end{equation}
Let us take $\overline{\rho}$ such that
\[1-C\overline{\rho}^{1/n}F[g]\geq\frac{1}{2}\]
so that, by \eqref{31} and \eqref{e} the thesis \eqref{LPS} follows for $\rho\leq \overline{\rho}$. For larger values of $\rho$ iniquality \eqref{LPS} is trivial.
\end{proof}
Now, we need a doubling inequality for $\hat{\nabla} u$. We start with recalling a doubling inequality for $u$ that corresponds to \cite[Theorem 1.2]{KLW2016}.

\begin{theorem}
Under assumption \ref{ipotesi-coeff}, there exists a positive constant $C$ such that for every $v\in H^1(B_{2R})$ solution to $\dive(\C\hat{\nabla}v)=0$  we have
\begin{equation}\label{doubu}
\int_{B_{2r}(x)}|v|^2dx\leq C\int_{B_{r}(x)}|v|^2dx
\end{equation}
for every $B_{2r}(x)\subset B_{R/2}$ and with $C$ depending on $\alpha_0$, $\beta_0$, $M$ and increasingly  on
\[F_{loc}=\frac{\|v\|_{L^2(B_{2R}\setminus B_{R})}}{\|v\|_{L^2(B_R\setminus B_{R/2})}}\] 
\end{theorem}

\begin{theorem}
Under assumptions \ref{ipotesi-dominio}, \ref{ipotesi-coeff} and \ref{ipotesi-g}, let $u$ be a solution to \eqref{def-u}. Then, for every $x_0\in \Omega_d$ and $0<r\leq d$,
\begin{equation} \label{trace:korn:poincare}
\int_{B_r(x_0)} |\hat{\nabla} u|^2\,dx \geq  C_d \left( \dfrac{r}{d}\right)^{K}\|g\|^2_{H^{1/2}(\partial\Omega)},
\end{equation} 
where $C_d$ and $K$ depend on $\alpha_0$, $\beta_0$, $M$, $r_0$, $M_0$, $|\Omega|$ and $K$ depends also on $F[g]$.
\end{theorem}
\begin{proof}
Let 
\[v=u-c_r-W_r(x-x_0)\]
where
\[c_r=\frac{1}{|B_r(x_0)|}\int_{B_r(x_0)}u\,dx\mbox{ and }
W_r=\frac{1}{|B_r(x_0)|}\int_{B_r(x_0)}\hat{\nabla}u\,dx.\]
Since function $v$ is still a solution of  equation
$\dive(\C_1 \hat{\nabla}v)=0$ in $\Omega$, by Caccioppoli inequality
(see \cite[Lemma 3.4]{AMR2004}) we have
\begin{equation*}
	\int_{B_{3r/2}(x_0)}|\nabla v|^2dx\leq \frac{C}{r^2}\int_{B_{2r}(x_0)}|v|^2dx,
\end{equation*}
where $C$ depends only on $\alpha_0$, $\beta_0$ and $M$,
hence,
 trivially,
\begin{equation}\label{cac}
	\int_{B_{3r/2}(x_0)}|\hat{\nabla} u|^2dx=\int_{B_{3r/2}(x_0)}|\hat{\nabla} v|^2dx\leq \frac{C}{r^2}\int_{B_{2r}(x_0)}| v|^2dx.
\end{equation}
By Korn inequality (see \cite[Lemma 3.5]{AMR2004}
\begin{equation}\label{korn}
	\int_{B_r(x_0)}|v|^2dx=\int_{B_r(x_0)}|u-c_r-W_r(x-x_0)|^2dx
	\leq Cr^2\int_{B_r(x_0)}|\hat{\nabla}u|^2dx.
\end{equation}
By \eqref{doubu}, \eqref{cac} and \eqref{korn} we have
\begin{equation}\label{1-3}
	\int_{B_{3r/2}(x_0)}|\hat{\nabla} u|^2dx\leq C \int_{B_{r}(x_0)}|\hat{\nabla} u|^2dx
\end{equation}
where $C$ depends on $\alpha_0$, $\beta_0$, $M$, $r_0$, $M_0$, $|\Omega|$ and increasingly on
\[F_{r,loc}=\frac{\|u-c_r-W_r(x-x_0)\|_{L^2(B_{2R}\setminus B_{R})}}{\|u-c_r-W_r(x-x_0)\|_{L^2(B_R\setminus B_{R/2})}}.\]
Now, we need to bound $F_{r,loc}$ from above independently of $r$.
First of all we notice that,
\[|c_r|\leq \|u\|_{L^\infty(B_r(x_0))}\leq \|u\|_{L^\infty(B_{2R})},\mbox{ and }
|W_r|\leq \|\nabla u\|_{L^\infty(B_r(x_0))}\leq\|\nabla u\|_{L^\infty(B_{2R})},\]
hence, 
by internal regularity estimates (see, for example \cite{C}) and \cite[Theorem 4.2, chap.3]{V}, 
\begin{equation}\label{pag4}
\begin{aligned}
\|u-c_r-W_r(x-x_0)\|_{L^2(B_{2R}\setminus B_{R})}&\leq C\left(
\|u\|_{L^\infty(B_{2R})}+\|\nabla u\|_{L^\infty(B_{2R})}\right)\\
&\leq C \|u\|_{H^1(\Omega)}\leq C\|g\|_{H^{1/2}(\partial\Omega)}.
\end{aligned}
\end{equation}
Let us now consider a ball $B_{r_1}(\overline{x})\subset B_{R}\setminus B_{R/2}$ with $r_1=\max\{d/5,R/4\}$ and notice that, by Caccioppoli inequality,
\begin{eqnarray}\label{pag5}
\int_{B_{R}\setminus B_{R/2}}|u-c_r-W_r(x-x_0)|^2dx&\geq&
	\int_{B_{r_1}(\overline{x})}|u-c_r-W_r(x-x_0)|^2dx\nonumber\\&\geq&
\frac{r_1^2}{C}\int_{B_{\frac{r_1}{2}}(\overline{x})}|\nabla(u-c_r-W_r(x-x_0))|^2dx\nonumber\\
&\geq& 
\frac{r_1^2}{C}\int_{B_{\frac{r_1}{2}}(\overline{x})}|\hat{\nabla}u|^2dx.
\end{eqnarray}
By \eqref{LPS},
\begin{equation*}	\int_{B_{\frac{r_1}{2}}(\overline{x})}|\hat{\nabla}u|^2dx\geq
C_{r_1}\int_\Omega |\hat{\nabla}u|^2dx.
\end{equation*}
Arguing as in the proof of Theorem \ref{LPSmallness}, by \eqref{servedopo},
we have
\begin{equation}\label{pag5-6}		\int_{B_{\frac{r_1}{2}}(\overline{x})}|\hat{\nabla}u|^2dx\geq
C_{r_1}\|g-(\overline{a}+\overline{W}x)\|^2_{H^{1/2}(\partial\Omega)}.
\end{equation}
and, hence,  by \eqref{pag4}, \eqref{pag5} and \eqref{pag5-6}
\begin{equation}\label{stimafreqloc}
	F_{r,loc}\leq C F[g]
\end{equation}
where $C$ does not depend on $r$.

Then, by \eqref{1-3}  and \eqref{stimafreqloc},
\begin{equation}\label{doubgrad}
	\int_{B_{3r/2}(x_0)}|\hat{\nabla} u|^2dx\leq C \int_{B_{r}(x_0)}|\hat{\nabla} u|^2dx
\end{equation}
where $C$ depends on $\alpha_0$, $\beta_0$, $M$, $r_0$, $M_0$, $|\Omega|$ and increasingly on $F[g]$.

Once the doubling inequality \eqref{doubgrad} for $\hat{\nabla}u$ is obtained, the polynomial rate of 
$\int_{B_{r}(x_0)}|\hat{\nabla} u|^2dx$ can be easily obtained by iteration (see \cite[Remark 4.11]{AMR2002} for a similar procedure).
\end{proof}
\section{Proof of Theorem \ref{mainth}}\label{proof}
Here we follow an argument already used in \cite{ADcFV2016} (see e.g. proof of Theorem 3.1).
Let us set again $\varphi=\mu_1-\mu_2$. By \eqref{DR1} we obtain 
\[
\int_{\Omega_d} |\varphi||\widehat{\nabla u}|^2\,dx
\leq
\varepsilon^2,
\]
where
\[\varepsilon^2=
C\left( 
\|\mu_1-\mu_2\|_{L^\infty(\partial\Omega)} 
+
\|u-v\|^{1/4}_{L^2(\Omega)}
\right)\]
Now, let $x_0\in \Omega_d$ be such that $|\varphi(x_0)|=\max_{\bar{\Omega}_d}|\varphi|$. Using the Lipschitz assumption on $\mu_1$, $\mu_2$ we obtain
$$
|\varphi(x_0)| \leq |\varphi(x)| + 2Mr
\qquad \forall x \in B_r(x_0),\, r\in (0,d].
$$ 

Multiplying both sides by $|\hat{\nabla} u|^2$ and integrating over $B_r(x_0)$ we get
\begin{align*}
|\varphi(x_0)|\int_{B_r(x_0)} |\hat{\nabla} u|^2\,dx
&
\leq
\int_{B_r(x_0)} |\varphi(x)||\hat{\nabla} u|^2\,dx 
+
2Mr\int_{B_r(x_0)} |\hat{\nabla} u|^2\,dx 
\\
&
\leq
\int_{\Omega} |\hat{\nabla} u|^2\,dx
+
2Mr \int_{B_r(x_0)} |\hat{\nabla} u|^2\,dx 
\\
&
\leq 
\varepsilon^2 + 2Mr \int_{B_r(x_0)} |\hat{\nabla} u|^2\,dx\,.
\end{align*}
Then,
$$
|\varphi(x_0)|
\leq 
\dfrac{\varepsilon^2}{\int_{B_r(x_0)} |\hat{\nabla} u|^2\,dx} 
+
2Mr.
$$
Now we use \eqref{trace:korn:poincare} and set 
$N_1=\left(C_d\|g\|^2_{H^{1/2}(\partial\Omega)}\right)^{-1}$,
$N_2=\log_2 K$, as to get
$$
|\varphi(x_0)|\leq N_1 \left( \dfrac{d}{r}\right)^{N_2} \varepsilon^2 + 2Mr
\qquad \forall r \in (0,d]\,.
$$
Finally, by setting $\lambda=\dfrac{r}{d}$ we obtain
$$
|\varphi(x_0)|\leq N_1 \lambda^{-N_2} \varepsilon^2 + 2Md\lambda
\qquad \forall \lambda \in (0,1]\,.
$$
Let 
$$
\bar{\lambda} = \left(\dfrac{ N_1\varepsilon^2}{2Md} \right)^{1/(N_2+1)}. 
$$
If $\bar{\lambda} \leq 1$ we choose $\lambda=\bar{\lambda}$ and get
\begin{equation} \label{lambda<1}
|\varphi(x_0)|
\leq
2  \left( N_1 \varepsilon^2 \right)^{1/(N_2+1)} \left( 2Md\right) ^{N_2/(N_2+1)}.
\end{equation}
If $\bar{\lambda}>1$ we immediately get
$$
|\varphi(x_0)|\leq 2M\leq 2M\left( \dfrac{N_1\varepsilon^2}{2Md}\right) ^{1/(N_2+1)}
$$
from which  together with \eqref{lambda<1} we obtain \eqref{main-result}.\qed

\end{document}